\documentclass[12pt]{article}
\usepackage{amsthm,geometry,amssymb,amsmath,enumerate,float,tikz,cite,setspace}
\newtheorem{theorem}{Theorem}
\newtheorem{lemma}[theorem]{Lemma}
\newtheorem{corollary}[theorem]{Corollary}

\allowdisplaybreaks

\title{Additive Tree $O(\rho\log n)$-Spanners from Tree Breadth $\rho$}
\author{Oliver Bendele and Dieter Rautenbach\\[3mm]
\normalsize Institute of Optimization and Operations Research, Ulm University, Germany\\
\texttt{$\{$oliver.bendele,dieter.rautenbach$\}$@uni-ulm.de}}
\date{}

\begin{document}
\maketitle
\onehalfspace

\begin{abstract}
The tree breadth ${\rm tb}(G)$ of a connected graph $G$ 
is the smallest non-negative integer $\rho$
such that $G$ has a tree decomposition 
whose bags all have radius at most $\rho$.
We show that, given a connected graph $G$ 
of order $n$ and size $m$,
one can construct 
in time $O(m\log n)$ 
an additive tree $O\big({\rm tb}(G)\log n\big)$-spanner of $G$,
that is, a spanning subtree $T$ of $G$ 
in which 
$d_T(u,v)\leq d_G(u,v)+O\big({\rm tb}(G)\log n\big)$
for every two vertices $u$ and $v$ of $G$.
This improves earlier results of Dragan and K\"{o}hler 
(Algorithmica 69 (2014) 884-905),
who obtained a multiplicative error of the same order,
and of Dragan and Abu-Ata
(Theoretical Computer Science 547 (2014) 1-17),
who achieved the same additive error with a collection of $O(\log n)$ trees.\\[2mm]
\end{abstract}

\section{Introduction}

In the present paper we show how to construct
in time $O(m\log n)$,
for a given connected graph $G$ of order $n$ and size $m$,
a tree spanner 
that approximates all distances 
up to some additive error of the form $O(\rho\log n)$,
where $\rho$ is the so-called tree breadth of $G$ \cite{drko}.
Our result improves a result of Dragan and K\"{o}hler \cite{drko}
who show that one can construct in time $O(m\log n)$
a multiplicative tree $O(\rho\log n)$-spanner 
for a given graph $G$ as above,
that is, we improve their multiplicative error 
to an additive one of the same order.
Our result also improves a result by Dragan and Abu-Ata \cite{drab} 
who show how to efficiently construct 
$O(\log n)$ collective additive tree $O(\rho\log n)$-spanners 
for a given graph $G$ as above.
Note that they obtain the same additive error bound
but require several spanning trees 
that respect this bound only collectively,
more precisely, for every pair of vertices, 
there is a tree in the collection 
that satisfies the distance condition for this specific pair.
Not restricting the spanners to trees allows better guarantees;
Dourisboure, Dragan, Gavoille, and Yan \cite{dodrgaya}, 
for instance, 
showed that every graph $G$ as above
has an additive $O(\rho)$-spanner with $O(\rho n)$ edges.
For more background on additive and multiplicative (collective) (tree) spanners
please refer to \cite{dodrgaya,drab,drko,krlemuprwa,caco,dryaco,drcokoxi}
and the references therein.

Before we come to our results in Section 2,
we collect some terminology and definitions.
We consider finite, simple, and undirected graphs.
Let $G$ be a connected graph.
The 
{\it vertex set},
{\it edge set},
{\it order}, and 
{\it size} of $G$ are denoted by 
$V(G)$,
$E(G)$,
$n(G)$, and $m(G)$,
respectively.
The {\it distance} in $G$ between two vertices $u$ and $v$ of $G$
is denoted by $d_G(u,v)$.
For a vertex $u$ of $G$ and a set $U$ of vertices of $G$,
the {\it distance} in $G$ between $u$ and $U$ is 
$$d_G(u,U)=\min\big\{ d_G(u,v):v\in U\big\},$$
and the {\it radius ${\rm rad}_G(U)$} of $U$ in $G$ is 
$$\min\big\{\max\big\{ d_G(u,v):v\in U\big\}:u\in V(G)\big\},$$
that is, it is the smallest radius of a ball around some vertex $u$ of $G$
that contains all of $U$.
Note that the vertex $u$ in the preceding minimum 
is not required to belong to $U$, and that all distances
are considered within $G$.

Let $H$ be a subgraph of $G$.
For a non-negative integer $k$, 
the subgraph $H$ is {\it $k$-additive} 
if 
\begin{eqnarray}\label{e1}
d_H(u,v)\leq d_G(u,v)+k
\end{eqnarray}
for every two vertices $u$ and $v$ of $H$.
If, additionally, the subgraph $H$ is {\it spanning},
that is, it has the same vertex set as $G$,
then $H$ is an {\it additive $k$-spanner} of $G$.
Furthermore, if, again additionally, the subgraph $H$ is a tree, 
then $H$ is an {\it additive tree $k$-spanner} of $G$.
Replacing the inequality (\ref{e1}) with 
$$d_H(u,v)\leq k\cdot d_G(u,v)$$
yields the notions of 
a {\it $k$-multiplicative} subgraph,
a {\it multiplicative $k$-spanner}, and
a {\it multiplicative tree $k$-spanner} of $G$, respectively.

For a tree $T$, let $B(T)$ be the set of vertices of $T$ 
of degree at least $3$ in $T$, the so-called {\it branch} vertices,
and let $L(T)$ be the set of leaves of $T$.

A {\it tree decomposition} of $G$ is a pair $\left(T,(X_t)_{t\in V(T)}\right)$,
where $T$ is a tree and $X_t$ is a set of vertices of $G$ for every vertex $t$ of $T$
such that 
\begin{itemize}
\item for every vertex $u$ of $G$, the set 
$\big\{ t\in V(T):u\in X_t\big\}$
induces a non-empty subtree of $T$, and
\item for every edge $uv$ of $G$, 
there is some vertex $t$ of $T$ such that $u$ and $v$ both belong to $X_t$.
\end{itemize}
The set $X_t$ is usually called the {\it bag} of $t$.
The maximum radius
$$\max\big\{{\rm rad}_G(X_t):t\in V(T)\big\}$$
of a bag of the tree decomposition
is the {\it breadth} of this decomposition,
and the {\it tree breadth ${\rm tb}(G)$} of $G$ \cite{drko}
is the minimum breadth of a tree decomposition of $G$.
While the tree breadth is an NP-hard parameter \cite{duleni},
one can construct in linear time,
for a given connected graph $G$,
a tree decomposition of breadth at most $3{\rm tb}(G)$ \cite{abdr},
cf.~also \cite{doga,drko,chdr} involving the related notion of {\it tree length}.

\section{Results}

For a tree $T$,
let ${\rm pbt}(T)$ be the maximum depth
of a perfect binary tree 
that is a topological minor of $T$.
In some sense ${\rm pbt}(T)$ quantifies how much $T$ differs from a path.

Our main result is the following.

\begin{theorem}\label{theorem1}
Given a connected graph $G$ of size $m$ and a tree decomposition 
$\left(T,(X_t)_{t\in V(T)}\right)$ of $G$ of breadth $\rho$,
one can construct 
in time $O\big(m\cdot {\rm pbt}(T)\big)$ 
an additive tree $8\rho\big(2{\rm pbt}(T)+1\big)$-spanner of $G$.
\end{theorem}
Some immediate consequences of Theorem \ref{theorem1} are the following.

\begin{corollary}\label{corollary1}
Given a connected graph $G$ of order $n$ and size $m$,
one can construct 
in time $O\big(m\log n\big)$ an 
additive tree 
$O\big({\rm tb}(G)\log n\big)$-spanner of $G$.
\end{corollary}
\begin{proof}
As observed towards the end of the introduction, given $G$,
one can construct in linear time 
a tree decomposition $\left(T,(X_t)_{t\in V(T)}\right)$ 
of $G$ of breadth at most $3{\rm tb}(G)$.
Possibly by contracting edges $st$ of $T$
with $X_s\subseteq X_t$,
we may assume that $n(T)\leq n$.
Since a perfect binary tree of depth $b$ has 
$2^{b+1}-1$ vertices, it follows that 
$2^{{\rm pbt}(T)+1}-1\leq n(T)\leq n$,
and, hence, 
$${\rm pbt}(T)\leq \log_2(n+1)-1.$$
Applying Theorem \ref{theorem1} allows to construct 
in time $O\big(m\cdot {\rm pbt}(T)\big)=O\big(m\log n\big)$
an additive tree $24{\rm tb}(G)\big(2\log_2(n+1)-1\big)$-spanner of $G$.
\end{proof}

\begin{corollary}\label{corollary2}
Given a connected graph $G$ of order $n$ and size $m$
and a multiplicative tree $k$-spanner $T$ of $G$,
one can construct 
in time $O\big(m n\big)$ an 
additive tree $O\big(k\log n\big)$-spanner of $G$.
\end{corollary}
\begin{proof}
For every vertex $u$ of $G$, let $X_u$ be the set 
containing all vertices $v$ of $G$ with $d_T(u,v)\leq \left\lceil\frac{k}{2}\right\rceil$.
Since $T$ is a multiplicative tree $k$-spanner,
it follows easily that $\left(T,(X_t)_{t\in V(T)}\right)$ 
is a tree decomposition of $G$ of breadth at most $\left\lceil\frac{k}{2}\right\rceil$,
cf.~also \cite{drko}.
Note that $(X_t)_{t\in V(T)}$ can be determined by $n$ 
breadth first searches, each of which requires $O(m)$ time.
Applying Theorem \ref{theorem1} allows to construct 
in time $O\big(m\cdot {\rm pbt}(T)\big)=O\big(m\log n\big)$
an additive tree $O\big(k\log n\big)$-spanner of $G$.
\end{proof}
Note that if the tree $T$ in Theorem \ref{theorem1} is a path,
then we obtain an additive tree $O(\rho)$-spanner.
Kratsch et al.~\cite{krlemuprwa} constructed a sequence of outerplanar chordal graphs $G_1,G_2,\ldots$,
which limit the extend to which Theorem \ref{theorem1} can be improved.
The graph $G_1$ is a triangle, and, for every positive integer $k$,
the graph $G_{k+1}$ arises from $G_k$ by adding, 
for every edge $uv$ of $G_k$ that contains a vertex of degree $2$ in $G_k$,
a new vertex $w$ that is adjacent to $u$ and $v$;
cf.~Figure \ref{fig:snowflake} for an illustration.
It is easy to see $n(G_k)=3\cdot 2^{k-1}$ and that ${\rm tb}(G_k)=1$ for every positive integer $k$,
in particular, we have $k-1=\log_2\left(\frac{n(G_k)}{3}\right)$.
Now, Kratsch et al.~showed that $G_k$ admits no additive tree $(k-1)$-spanner,
that is, the graph $G_k$ admits no additive tree ${\rm tb}(G_k)\log_2\left(\frac{n(G_k)}{3}\right)$-spanner.

\tikzstyle{sfvertex}=[circle, draw, fill=black!100,inner sep=0pt, minimum width=3pt]
\begin{figure}[H]\centering	
\begin{tikzpicture}[thick,scale=0.5]
\draw \foreach \x in {30,150,270}{
(\x:2) node[sfvertex] {}  -- (\x+120:2)
(\x:2) node[sfvertex] {} -- (\x-60:3)
(\x:2) node[sfvertex] {} -- (\x+60:3)
(\x+60:3) node[sfvertex] {}
};
\end{tikzpicture}\hspace{1.5cm}
\begin{tikzpicture}[thick,scale=0.5]
\draw \foreach \x in {30,150,270}{
(\x:2) node[sfvertex] {}  -- (\x+120:2)
(\x:2) node[sfvertex] {} -- (\x-60:3)
(\x:2) node[sfvertex] {} -- (\x+60:3)
(\x-60:3) node[sfvertex] {} -- (\x-90:4) node[sfvertex] {}
(\x+60:3) node[sfvertex] {} -- (\x+90:4) node[sfvertex]{}
(\x:2) node[sfvertex] {}  -- (\x+30:4)
(\x:2) node[sfvertex] {} -- (\x-30:4)};
\end{tikzpicture}
\caption{The graphs $G_2$ and $G_3$.}\label{fig:snowflake}
\end{figure}
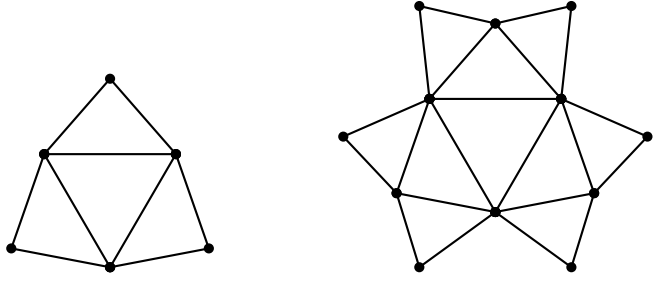
Our proof of Theorem \ref{theorem1} relies on four lemmas.
The first is a simple consequence of elementary properties of breadth first search

\begin{lemma}\label{lemma1}
Given a connected graph $G$ of size $m$,
a subtree $S$ of $G$,
and a set $U$ of vertices of $G$,
one can construct in time $O(m)$
a subtree $S'$ of $G$ containing $S$ 
as well as all vertices from $U$ such that 
\begin{enumerate}[(i)]
\item $d_{S'}(u,V(S))=d_G(u,V(S))$
for every vertex $u$ in $U$, and
\item $L(S')\subseteq L(S)\cup U$.
\end{enumerate}
\end{lemma}
\begin{proof}
The tree $S'$ with the desired properties can be obtained as follows:
\begin{itemize}
\item Construct the graph $G'$ from $G$ by contracting $S$ to a single vertex $r$.
\item Construct a breadth first search tree $T$ of $G'$ rooted in $r$.
\item Construct the graph $T'$ from $T$ by uncontracting $r$ back to $S$.
\item Choose $S'$ as the minimal subtree of $T'$ 
that contains $S$ as well as all vertices from $U$.
\end{itemize}
Since $T$ is a breadth first search tree, property (i) follows.
Furthermore, by construction, 
the set of leaves of $S'$ is contained in $L(S)\cup U$,
that is, property (ii) follows.
The running time follows easily from the running time 
of breadth first search;
in fact, the contraction of $S$ to $r$ can be handled 
implicitly within a suitably adapted breadth first search.
\end{proof}
The following lemma was inspired by Lemma 2.2 in \cite{krlemuprwa}.
It will be useful to complete the construction 
of our additive tree spanner
starting from a suitable subtree.

\begin{lemma}\label{lemma2}
Given a connected graph $G$ of size $m$
and a $\rho$-additive subtree $S$ of $G$
such that $d_G(u,V(S))\leq \rho'$ for every vertex $u$ of $G$,
one can construct in time $O(m)$
an additive tree $(\rho+4\rho')$-spanner of $G$.
\end{lemma}
\begin{proof}
Let $S'$ be the spanning tree of $G$ 
obtained by applying Lemma \ref{lemma1} 
to $G$, $S$, and $V(G)\setminus V(S)$ as the set $U$.
We claim that $S'$ has the desired properties.
Therefore, let $u$ and $v$ be any two vertices of $G$.
Let $u'$ be the vertex of $S$ closest to $u$ within $S'$,
and define $v'$ analogously.
Clearly, we have that 
$d_{S'}(u,u')=d_G(u,u')\leq \rho'$,
$d_{S'}(v,v')=d_G(v,v')\leq \rho'$, and 
$d_{S'}(u',v')=d_S(u',v')\leq d_G(u',v')+\rho$.
By several applications of the triangle inequality, 
we obtain
\begin{eqnarray*}
d_{S'}(u,v) & = & d_{S'}(u,u')+d_S(u',v')+d_{S'}(v',v)\\
& \leq & \rho'+d_G(u',v')+\rho+\rho'\\
& \leq & d_G(u',u)+d_G(u,v)+d_G(v,v')+\rho+2\rho'\\
& \leq & d_G(u,v)+\rho+4\rho',
\end{eqnarray*}
which completes the proof.
\end{proof}
Our next lemma states that ${\rm pbt}(T)$ can easily be determined
for a given tree $T$, by constructing a suitable finite sequence
\begin{eqnarray}\label{e2}
T_0\supset T_1\supset T_2\supset\ldots\supset T_{{\rm d}(T)}
\end{eqnarray}
of nested trees.
The construction of this sequence is also important 
for the proof of our main technical lemma,
cf.~Lemma \ref{lemma4} below.
The sequence starts with $T_0$ equal to $T$.
Now, suppose that $T_i$ has been defined 
for some non-negative integer $i$.
If $B(T_i)$ is not empty, 
then let $T_{i+1}$ be the minimal subtree of $T_i$ 
that contains all vertices from $B(T_i)$,
and continue the construction.
Note that in this case
$$B(T_i)=B(T_{i+1})\cup L(T_{i+1}).$$ 
Otherwise, if $B(T_i)$ is empty, then $T_i$ 
is a path of some length $\ell$.
If $\ell\geq 3$, then 
let $T_{i+1}$ be the tree containing exactly one 
internal vertex of $T_i$ as its only vertex,
and let ${\rm d}(T)=i+1$.
Finally, if $\ell\leq 2$, then let ${\rm d}(T)=i$.
Once ${\rm d}(T)$ has been defined, 
the construction of the sequence (\ref{e2}) terminates.
See Figure \ref{fig.seq} for an illustration.

\begin{figure}[H]
\begin{center}
\unitlength 0.75mm 
\linethickness{0.4pt}
\ifx\plotpoint\undefined\newsavebox{\plotpoint}\fi 
\begin{picture}(76,46)(0,0)
\put(24,26){\circle*{1.5}}
\put(36,41){\circle*{1.5}}
\put(51,32){\circle*{1.5}}
\put(40,28){\circle*{1.5}}
\put(55,19){\circle*{1.5}}
\put(10,26){\circle*{1.5}}
\put(18,36){\circle*{1.5}}
\put(45,41){\circle*{1.5}}
\put(41,19){\circle*{1.5}}
\put(27,16){\circle*{1.5}}
\put(27,6){\circle*{1.5}}
\put(44,10){\circle*{1.5}}
\put(37,10){\circle*{1.5}}
\put(10,26){\line(4,5){8}}
\put(18,36){\line(3,-5){6}}
\multiput(24,26)(.03370787,-.11235955){89}{\line(0,-1){.11235955}}
\put(27,16){\line(5,-3){10}}
\put(37,10){\line(1,0){7}}
\put(36,41){\line(-4,-5){12}}
\put(36,41){\line(1,0){9}}
\put(45,41){\line(2,-3){6}}
\multiput(51,32)(.033613445,-.109243697){119}{\line(0,-1){.109243697}}
\multiput(51,32)(-.092436975,-.033613445){119}{\line(-1,0){.092436975}}
\multiput(40,28)(.0333333,-.3){30}{\line(0,-1){.3}}
\multiput(36,10)(-.075630252,-.033613445){119}{\line(-1,0){.075630252}}
\put(2,20){\circle*{1.5}}
\put(14,15){\circle*{1.5}}
\put(6,7){\circle*{1.5}}
\put(10,26){\line(-4,-3){8}}
\multiput(10,26)(.033613445,-.092436975){119}{\line(0,-1){.092436975}}
\put(14,15){\line(-1,-1){8}}
\put(56,45){\circle*{1.5}}
\put(65,35){\circle*{1.5}}
\put(75,34){\circle*{1.5}}
\put(71,44){\circle*{1.5}}
\multiput(75,34)(-.3333333,.0333333){30}{\line(-1,0){.3333333}}
\put(65,35){\line(2,3){6}}
\multiput(65,35)(-.0337078652,.0374531835){267}{\line(0,1){.0374531835}}
\multiput(56,45)(-.092436975,-.033613445){119}{\line(-1,0){.092436975}}
\end{picture}
\linethickness{0.4pt}
\ifx\plotpoint\undefined\newsavebox{\plotpoint}\fi 
\begin{picture}(56,46)(0,0)
\put(24,26){\circle*{1.5}}
\put(36,41){\circle*{1.5}}
\put(51,32){\circle*{1.5}}
\put(10,26){\circle*{1.5}}
\put(18,36){\circle*{1.5}}
\put(45,41){\circle*{1.5}}
\put(27,16){\circle*{1.5}}
\put(37,10){\circle*{1.5}}
\put(10,26){\line(4,5){8}}
\put(18,36){\line(3,-5){6}}
\multiput(24,26)(.03370787,-.11235955){89}{\line(0,-1){.11235955}}
\put(27,16){\line(5,-3){10}}
\put(36,41){\line(-4,-5){12}}
\put(36,41){\line(1,0){9}}
\put(45,41){\line(2,-3){6}}
\put(56,45){\circle*{1.5}}
\put(65,35){\circle*{1.5}}
\multiput(65,35)(-.0337078652,.0374531835){267}{\line(0,1){.0374531835}}
\multiput(56,45)(-.092436975,-.033613445){119}{\line(-1,0){.092436975}}
\end{picture}
\linethickness{0.4pt}
\ifx\plotpoint\undefined\newsavebox{\plotpoint}\fi 
\begin{picture}(30,43)(0,0)
\put(24,26){\circle*{1.5}}
\put(36,41){\circle*{1.5}}
\put(45,41){\circle*{1.5}}
\put(36,41){\line(-4,-5){12}}
\put(36,41){\line(1,0){9}}
\end{picture}
\linethickness{0.4pt}
\ifx\plotpoint\undefined\newsavebox{\plotpoint}\fi 
\begin{picture}(38,43)(0,0)
\put(36,41){\circle*{1.5}}
\end{picture}
\end{center}
\vspace{-5mm}\caption{A sequence $T_0\subset T_1\subset T_2\subset T_3$.}\label{fig.seq}
\end{figure}
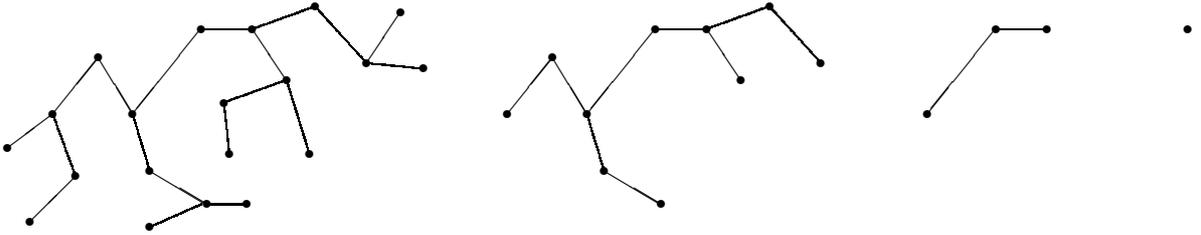

\begin{lemma}\label{lemma3}
${\rm pbt}(T)={\rm d}(T)$ for every tree $T$.
\end{lemma}
\begin{proof}
The proof is by induction on ${\rm d}(T)$.
If ${\rm d}(T)=0$, the statement is trivial.
Now, let ${\rm d}(T)\geq 1$.
The construction of (\ref{e2}) immediately implies 
$${\rm d}(T)={\rm d}(T_1)+1.$$
Note that $T$ arises from $T_1$ by attaching 
at least two new paths to each leaf of $T_1$
so that each leaf of $T_1$ becomes a branch vertex of $T$.
Therefore, if $S_1$ is a subtree of $T_1$
that is a subdivision of a perfect binary tree,
then one can first extend $S_1$ in such a way that 
all leaves of $S_1$ are also leaves of $T_1$,
and then one can grow one further level 
to the subdivided binary tree 
by attaching two new paths to each leaf of $S_1$ 
using edges in $E(T)\setminus E(T_1)$.
This implies ${\rm pbt}(T)\geq {\rm pbt}(T_1)+1$.
Conversely,
if $S$ is a subtree of $T$
that is a subdivision of a perfect binary tree,
then $S\cap T_1$ contains 
a subdivision of a perfect binary tree
whose depth is one less,
that is, we have ${\rm pbt}(T_1)\geq {\rm pbt}(T)-1$.
Altogether, by induction, we obtain
$${\rm pbt}(T)={\rm pbt}(T_1)+1={\rm d}(T_1)+1={\rm d}(T),$$
which completes the proof.
\end{proof}
The following is our core technical lemma.

\begin{lemma}\label{lemma4}
Given a connected graph $G$ of size $m$ and a tree decomposition 
$\left(T,(X_t)_{t\in V(T)}\right)$ of $G$ of breadth $\rho$,
one can construct 
in time $O(m\cdot {\rm d}(T))$ a $16\rho\cdot {\rm d}(T)$-additive subtree $S$ of $G$ 
intersecting each bag of the given tree-decomposition.
\end{lemma}
\begin{proof}
Let the sequence 
$T_0\supset T_1\supset T_2\supset\ldots\supset T_{{\rm d}(T)}$
be as in (\ref{e2}),
and let $d={\rm d}(T)$.
For $i$ from $d$ down to $0$, 
we explain how to recursively construct 
a subtree $S_i$ of $G$ such that
\begin{enumerate}[(i)]
\item $S_i$ contains a vertex from bag $X_t$ 
for every vertex $t$ of $T_i$,
\item for every two distinct leaves $u$ and $v$ of $S_i$,
there are two distinct vertices $s$ and $t$ 
of $T_i$ that belong to $B(T_i)\cup L(T_i)$
such that $u\in X_s$ and $v\in X_t$, and
\item $S_i$ is $16\rho(d-i)$-additive.
\end{enumerate}
Note that $S_0$ is a subtree of $G$ with the desired properties.

First, we consider $i=d$.
The tree $T_d$ has order at most $2$,
and, since $G$ is connected,
there is a vertex $u$ of $G$ 
that belongs to all bags $X_t$ with $t\in V(T_d)$.
Let $S_d$ be the subtree of $G$ containing only the vertex $u$.
Since $S_d$ has order $1$, and all vertices of $T_d$ are leaves, 
properties (ii) and (iii) are trivial for $S_d$, 
and property (i) follows from the choice of $u$.
See Figure \ref{figinS} for an illustration.

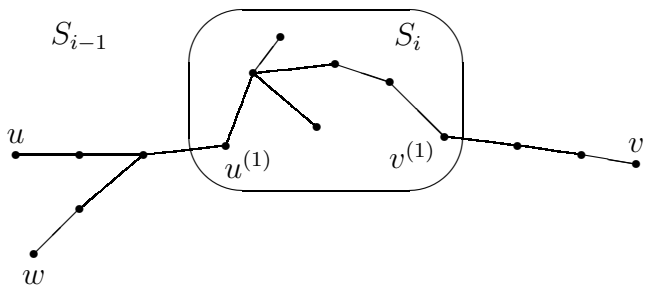
\begin{figure}[H]
\begin{center}
\unitlength 1.2mm 
\linethickness{0.4pt}
\ifx\plotpoint\undefined\newsavebox{\plotpoint}\fi 
\begin{picture}(71,30)(0,0)
\put(36,20){\oval(30,20)[]}
\put(25,15){\circle*{1}}
\put(28,23){\circle*{1}}
\put(35,17){\circle*{1}}
\put(43,22){\circle*{1}}
\put(49,16){\circle*{1}}
\put(37,24){\circle*{1}}
\multiput(28,23)(-.03370787,-.08988764){89}{\line(0,-1){.08988764}}
\multiput(28,23)(.039325843,-.033707865){178}{\line(1,0){.039325843}}
\multiput(28,23)(.3,.0333333){30}{\line(1,0){.3}}
\put(37,24){\line(3,-1){6}}
\put(43,22){\line(1,-1){6}}
\put(57,15){\circle*{1}}
\put(64,14){\circle*{1}}
\put(70,13){\circle*{1}}
\put(16,14){\circle*{1}}
\put(9,14){\circle*{1}}
\put(2,14){\circle*{1}}
\put(9,8){\circle*{1}}
\put(4,3){\circle*{1}}
\put(2,14){\line(1,0){14}}
\multiput(16,14)(-.039325843,-.033707865){178}{\line(-1,0){.039325843}}
\put(9,8){\line(-1,-1){5}}
\multiput(16,14)(.3,.0333333){30}{\line(1,0){.3}}
\put(9,27){\makebox(0,0)[cc]{$S_{i-1}$}}
\put(45,27){\makebox(0,0)[cc]{$S_i$}}
\multiput(49,16)(.2666667,-.0333333){30}{\line(1,0){.2666667}}
\multiput(57,15)(.2333333,-.0333333){30}{\line(1,0){.2333333}}
\put(64,14){\line(6,-1){6}}
\put(27.5,13){\makebox(0,0)[cc]{$u^{(1)}$}}
\put(45.5,14){\makebox(0,0)[cc]{$v^{(1)}$}}
\put(2,16){\makebox(0,0)[cc]{$u$}}
\put(4,0.5){\makebox(0,0)[cc]{$w$}}
\put(70,15){\makebox(0,0)[cc]{$v$}}
\put(31,27){\circle*{1}}
\put(31,27){\line(-3,-4){3}}
\end{picture}
\end{center}
\caption{Extending $S_i$ to $S_{i-1}$, 
and possible positions of the vertices $u$, $v$, $u^{(1)}$, and $v^{(1)}$
explained below.}\label{figinS}
\end{figure}
Now, suppose that $S_i$ has already been defined 
for some integer $i$ with $d\geq i>0$. 
We explain how to construct $S_{i-1}$.
Therefore, let $U$ be an inclusion-wise minimal set of vertices 
intersecting every bag $X_t$
such that $t$ is a leaf of $T_{i-1}$
for which $S_i$ does not contain a vertex from $X_t$.
Let $S_{i-1}$ arise by applying 
Lemma \ref{lemma1}
to $G$, $S_i$ as $S$, and $U$.
By construction,
the subgraph $S_{i-1}$ of $G$ is connected
and contains a vertex from every bag $X_t$
such that $t$ is a leaf of $T_{i-1}$.
Since $G$ is connected, 
basic properties of tree decompositions
imply that $S_{i-1}$ satisfies property (i),
that is, the vertex set of $S_{i-1}$ intersects every bag of $T_{i-1}$.

Next, we verify property (ii) for $S_{i-1}$.
Therefore, let $u$ and $v$ be two distinct leaves of $S_{i-1}$.
If $u$ and $v$ are also leaves of $S_i$,
then property (ii) for $S_{i-1}$ follows
from property (ii) for $S_i$
using $B(T_i)\cup L(T_i)=B(T_{i-1})$.
If $u$ is a leaf of $S_i$ and $v$ is not,
then, by Lemma \ref{lemma1}(ii),
we have $v\in U$.
By property (ii) for $S_i$,
the vertex $u$ belongs to a bag $X_s$ 
such that $s\in B(T_i)\cup L(T_i)=B(T_{i-1})$,
and, by the choice of $U$,
the vertex $v$ belongs to a bag $X_t$
such that $t$ is a leaf of $T_{i-1}$
and $S_i$ contains no vertex from $X_t$.
In particular, we have that $u\not\in X_t$,
which implies that $s$ and $t$ are distinct,
that is, property (ii) holds also in this case.
Finally, suppose that $u$ and $v$ 
are both leaves of $S_{i-1}$ but not of $S_i$.
The choice of $U$ as minimal with respect to inclusion
implies that property (ii) holds also in this final case.
Note that $X_s$ is allowed to contain $v$
and that $X_t$ is allowed to contain $u$
in property (ii).

Finally, we verify the crucial property (iii) for $S_{i-1}$.
Therefore, let $u$ and $v$ be two distinct vertices of $S_{i-1}$.
It is easy to see that in order to verify that $S_{i-1}$
is $16\rho(d-(i-1))$-additive,
it suffices to consider the case where $u$ and $v$ are leaves of $S_{i-1}$.
In fact, if (iii) is violated for $u$ and $v$, that is, 
we have $d_{S_{i-1}}(u,v)>d_G(u,v)+16\rho(d-(i-1))$, 
then the path in $S_{i-1}$ between $u$ and $v$ is contained 
in some path in $S_{i-1}$ between 
the two leaves $\tilde{u}$ and $\tilde{v}$ of $S_{i-1}$, and
\begin{eqnarray*}
d_{S_{i-1}}(\tilde{u},\tilde{v})
&=&d_{S_{i-1}}(\tilde{u},u)
+d_{S_{i-1}}(u,v)
+d_{S_{i-1}}(v,\tilde{v})\\
&>&d_G(\tilde{u},u)
+d_G(u,v)+16\rho(d-(i-1))+d_G(v,\tilde{v})\\
&\geq &
d_G(\tilde{u},\tilde{v})+16\rho(d-(i-1)),
\end{eqnarray*}
that is, the two leaves $\tilde{u}$ and $\tilde{v}$
also violate (iii).
Hence, we may assume that $u$ and $v$ are leaves of $S_{i-1}$.
Let $P$ be a shortest path in $G$ between $u$ and $v$,
and let $P_{i-1}$ be the path in $S_{i-1}$ between $u$ and $v$.
Let $u^{(1)}$ be the vertex of $S_i$ that is closest within $S_{i-1}$ to $u$,
and define $v^{(1)}$ analogously.
See Figure \ref{figinS} for an illustration.
By Lemma \ref{lemma1}(i), we have 
\begin{eqnarray*}
d_G\left(u,u^{(1)}\right) & = & d_{S_{i-1}}\left(u,u^{(1)}\right)\mbox{ and }\\
d_G\left(v,v^{(1)}\right) & = & d_{S_{i-1}}\left(v,v^{(1)}\right).
\end{eqnarray*}
By (ii) for $S_{i-1}$,
there are two distinct vertices $s$ and $t$ of $T_{i-1}$
that belong to $B(T_{i-1})\cup L(T_{i-1})$ such that $u\in X_s$ and $v\in X_t$.
Let $T'$ be the subgraph of $T_{i-1}$ that is induced 
by the set of all vertices $r$ of $T_{i-1}$ 
for which $S_i$ contains a vertex from the bag $X_r$.
Since $S_i$ is connected, 
it follows from basic properties of tree decompositions
that $T'$ is a subtree of $T_{i-1}$.
Since $B(T_{i-1})=B(T_i)\cup L(T_i)\subseteq V(T_i)$
and, by construction of $T_i$ from $T_{i-1}$, 
the path in $T_{i-1}$ between any two distinct leaves of $T_{i-1}$
contains a vertex of $T_i$,
property (i) for $S_i$ implies that $T'$ 
contains a vertex from the path $Q$ in $T_{i-1}$ between $s$ and $t$.
Let $s'$ be the vertex of $T'$ on $Q$ that is closest within $T_{i-1}$ to $s$.
By the definition of $T'$,
there is a vertex $u^{(2)}$ of $S_i$ that belongs to $X_{s'}$.
See Figure \ref{figinT} for an illustration.

\begin{figure}[H]
\begin{center}
\unitlength 1mm 
\linethickness{0.4pt}
\ifx\plotpoint\undefined\newsavebox{\plotpoint}\fi 
\begin{picture}(107.5,39.5)(0,0)
\put(7,7){\circle*{1}}
\put(17,7){\circle*{1}}
\put(27,7){\circle*{1}}
\put(37,7){\circle*{1}}
\put(47,7){\circle*{1}}
\put(57,7){\circle*{1}}
\put(67,7){\circle*{1}}
\put(77,7){\circle*{1}}
\put(87,7){\circle*{1}}
\put(97,7){\circle*{1}}
\put(107,7){\circle*{1}}
\put(7,7){\line(1,0){100}}
\put(52,16.5){\oval(40,25)[]}
\put(37,10){\makebox(0,0)[cc]{$s'$}}
\put(68,10){\makebox(0,0)[cc]{$t'$}}
\put(49,13){\circle*{1}}
\put(43,18){\circle*{1}}
\put(50,23){\circle*{1}}
\put(33,38){\circle*{1}}
\put(36,23){\circle*{1}}
\put(64,24){\circle*{1}}
\put(47,39){\circle*{1}}
\put(55,19){\circle*{1}}
\put(38,34){\circle*{1}}
\put(64,17){\circle*{1}}
\put(47,32){\circle*{1}}
\put(73,34){\circle*{1}}
\put(58,37){\circle*{1}}
\put(33,38){\line(5,-4){5}}
\multiput(38,34)(-.03333333,-.18333333){60}{\line(0,-1){.18333333}}
\multiput(36,23)(.046979866,-.033557047){149}{\line(1,0){.046979866}}
\put(43,18){\line(6,-5){6}}
\put(49,13){\line(-1,-3){2}}
\multiput(73,34)(-.0337078652,-.0374531835){267}{\line(0,-1){.0374531835}}
\multiput(64,24)(-.033707865,.073033708){178}{\line(0,1){.073033708}}
\put(50,23){\line(5,-4){5}}
\put(55,19){\line(-1,-1){6}}
\multiput(64,17)(-.033653846,-.048076923){208}{\line(0,-1){.048076923}}
\put(47,39){\line(0,-1){7}}
\put(47,32){\line(1,-3){3}}
\put(15,25){\makebox(0,0)[cc]{in $T_{i-1}$}}
\put(70,0){\makebox(0,0)[cc]{$Q$ ($s$-$t$-path in $T_{i-1}$)}}
\put(7,10){\makebox(0,0)[cc]{$s$}}
\put(107,10){\makebox(0,0)[cc]{$t$}}
\put(107,2){\makebox(0,0)[cc]{($v\in X_t$)}}
\put(7,2){\makebox(0,0)[cc]{($u\in X_s$)}}
\put(68,21){\makebox(0,0)[cc]{$T'$}}
\multiput(64,24)(-.060402685,-.033557047){149}{\line(-1,0){.060402685}}
\end{picture}
\end{center}
\caption{The path $Q$ in $T_{i-1}$ between $s$ and $t$,
the subtree $T'$ of $T_{i-1}$ intersecting $Q$,
and the vertices $s'$ and $t'$.}\label{figinT}
\end{figure}
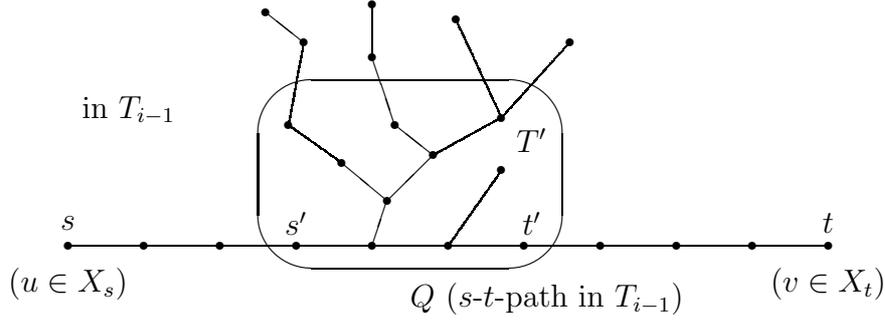
Basic properties of tree decompositions imply
that $X_{s'}$ contains a vertex from the path $P$ as well as from the path $P_{i-1}$.
Let $u^{(3)}$ be a vertex in $X_{s'}\cap V(P)$,
and let $u^{(4)}$ be the first vertex on the path $P_{i-1}$, 
when traversed from $u$ towards $v$,
that belongs to $X_{s'}$.
See Figure \ref{figinG} for an illustration.

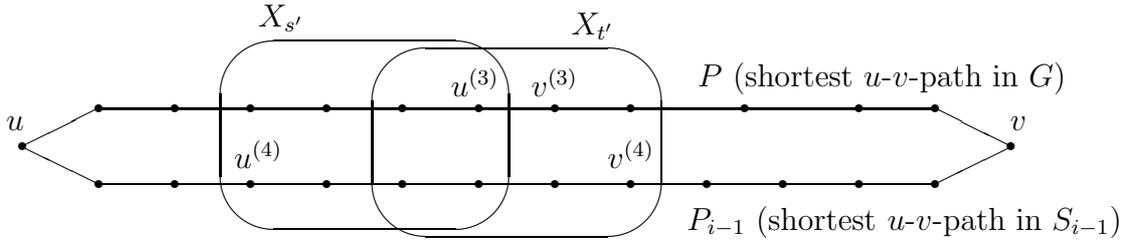
\begin{figure}[H]
\begin{center}
\unitlength 1mm 
\linethickness{0.4pt}
\ifx\plotpoint\undefined\newsavebox{\plotpoint}\fi 
\begin{picture}(135,32)(0,0)
\put(4,14){\circle*{1}}
\put(14,19){\circle*{1}}
\put(24,19){\circle*{1}}
\put(34,19){\circle*{1}}
\put(44,19){\circle*{1}}
\put(54,19){\circle*{1}}
\put(64,19){\circle*{1}}
\put(74,19){\circle*{1}}
\put(84,19){\circle*{1}}
\put(99,19){\circle*{1}}
\put(114,19){\circle*{1}}
\put(124,19){\circle*{1}}
\put(14,9){\circle*{1}}
\put(24,9){\circle*{1}}
\put(34,9){\circle*{1}}
\put(44,9){\circle*{1}}
\put(54,9){\circle*{1}}
\put(64,9){\circle*{1}}
\put(74,9){\circle*{1}}
\put(84,9){\circle*{1}}
\put(94,9){\circle*{1}}
\put(104,9){\circle*{1}}
\put(114,9){\circle*{1}}
\put(124,9){\circle*{1}}
\put(134,14){\circle*{1}}
\put(4,14){\line(2,1){10}}
\put(14,19){\line(1,0){110}}
\put(124,19){\line(2,-1){10}}
\put(134,14){\line(-2,-1){10}}
\put(124,9){\line(-1,0){110}}
\put(14,9){\line(-2,1){10}}
\put(49,15.5){\oval(38,25)[]}
\put(69,14.5){\oval(38,25)[]}
\put(39,32){\makebox(0,0)[cc]{}}
\put(38,31){\makebox(0,0)[cc]{$X_{s'}$}}
\put(79,30){\makebox(0,0)[cc]{$X_{t'}$}}
\put(63.5,22.5){\makebox(0,0)[cc]{$u^{(3)}$}}
\put(35,13){\makebox(0,0)[cc]{$u^{(4)}$}}
\put(74,22.5){\makebox(0,0)[cc]{$v^{(3)}$}}
\put(3,17){\makebox(0,0)[cc]{$u$}}
\put(135,17){\makebox(0,0)[cc]{$v$}}
\put(117,23){\makebox(0,0)[cc]{$P$ (shortest $u$-$v$-path in $G$)}}
\put(120,4){\makebox(0,0)[cc]{$P_{i-1}$ (shortest $u$-$v$-path in $S_{i-1}$)}}
\put(84,13){\makebox(0,0)[cc]{$v^{(4)}$}}
\end{picture}
\end{center}
\caption{The shortest paths $P$ in $G$ and $P_{i-1}$ in $S_{i-1}$
between $u$ and $v$, 
their intersection with the bags $X_{s'}$ and $X_{t'}$,
the vertices $u^{(4)}$ and $v^{(4)}$,
and possible positions of $u^{(3)}$ and $v^{(3)}$.}\label{figinG}
\end{figure}
Suppose, for a contradiction, 
that $u^{(1)}$ is distinct from $u^{(4)}$, 
and that $u^{(1)}$ lies closer to $u$ on $P_{i-1}$ than $u^{(4)}$.
In this case, the choices of $u^{(1)}$ and $u^{(4)}$ imply 
that $u^{(1)}$ lies in some bag $X_r$ for a vertex $r$ of $T'$ distinct from $s'$,
and that $u^{(1)}$ does not lie in $X_{s'}$.
Since $s'$ separates $s$ from $r$ in $T_{i-1}$,
basic properties of tree decompositions imply
that $P_{i-1}$ contains a vertex from $X_{s'}$
that is strictly closer to $u$ than $u^{(4)}$,
contradicting the choice of $u^{(4)}$.
Hence, either $u^{(1)}$ equals $u^{(4)}$, 
or $u^{(4)}$ lies closer to $u$ on $P_{i-1}$ than $u^{(1)}$.

Since $u^{(2)}$, $u^{(3)}$, and $u^{(4)}$ all belong to the bag $X_{s'}$,
which is of radius at most $\rho$, 
the pairwise distances of these three vertices within $G$ are at most $2\rho$.
If $d_G\left(u^{(1)},u^{(4)}\right)>2\rho$,
then connecting $u$ to $u^{(4)}$ via $P_{i-1}$,
and connecting $u^{(4)}$ to $S_i$ via a shortest path in $G$,
which is of length at most $2\rho$ in view of $u^{(2)}$,
yields a contradiction to Lemma \ref{lemma1}(i).
Hence, we have 
$$d_G\left(u^{(1)},u^{(4)}\right)\leq 2\rho,$$
and, thus, we obtain
\begin{eqnarray*}
d_G\left(u^{(1)},u^{(3)}\right) 
& \leq & 
d_G\left(u^{(1)},u^{(4)}\right)+d_G\left(u^{(4)},u^{(3)}\right)
\leq 4\rho.
\end{eqnarray*}
Now, let $t'$ be the vertex of $T'$ on $Q$
that is closest within $T_{i-1}$ to $t$.
See Figure \ref{figinT} for an illustration.
Clearly, the vertex $t'$ lies on the subpath of $Q$ between $s'$ and $t$.
Since $u^{(3)}\in X_{s'}$ and $v\in X_t$,
basic properties of tree decompositions imply
that the subpath of $P$ between $u^{(3)}$ and $v$
contains a vertex $v^{(3)}$ of $X_{t'}$.
See Figure \ref{figinG} for an illustration.
Choosing $v^{(2)}$ and $v^{(4)}$ in a symmetric way,
and arguing similarly as above,
we obtain 
\begin{eqnarray*}
d_G\left(v^{(1)},v^{(3)}\right) \leq 4\rho.
\end{eqnarray*}
By property (iii) for $S_i$, we have
\begin{eqnarray*}
d_{S_i}\left(u^{(1)},v^{(1)}\right) \leq d_G\left(u^{(1)},v^{(1)}\right)
+16\rho(d-i).
\end{eqnarray*}
Note that the vertices $u$, $u^{(3)}$, $v^{(3)}$, and $v$ appear in this order on $P$.
Altogether, by multiple applications of the triangle inequality, we obtain
that 
\begin{eqnarray*}
d_{S_{i-1}}\left(u,v\right)
&= & d_{S_{i-1}}\left(u,u^{(1)}\right)
+d_{S_i}\left(u^{(1)},v^{(1)}\right)
+d_{S_{i-1}}\left(v^{(1)},v\right)\\
& = & 
d_G\left(u,u^{(1)}\right)
+d_{S_i}\left(u^{(1)},v^{(1)}\right)
+d_G\left(v^{(1)},v\right)\\
& \leq & 
d_G\left(u,u^{(1)}\right)
+d_G\left(u^{(1)},v^{(1)}\right)+16\rho(d-i)
+d_G\left(v^{(1)},v\right)\\
& \leq & 
d_G\left(u,u^{(3)}\right)+d_G\left(u^{(3)},u^{(1)}\right)\\
&&+d_G\left(u^{(1)},u^{(3)}\right)+d_G\left(u^{(3)},v^{(3)}\right)+d_G\left(v^{(3)},v^{(1)}\right)+16\rho(d-i)\\
&&+d_G\left(v^{(1)},v^{(3)}\right)+d_G\left(v^{(3)},v\right)\\
& \leq & 
d_G\left(u,u^{(3)}\right)+4\rho
+4\rho+d_G\left(u^{(3)},v^{(3)}\right)+4\rho
+16\rho(d-i)+4\rho+d_G\left(v^{(3)},v\right)\\
& = & 
d_G\left(u,v\right)+16\rho(d-(i-1)),
\end{eqnarray*}
which completes the proof of property (iii) for $S_{i-1}$.

We proceed to the running time of the described procedure.
Clearly, the sequence as in (\ref{e2})
can be determined in time $O(m\cdot {\rm d}(T))$,
and the tree $S_d$ can be obtained in time $O(m(G))$.
By Lemma \ref{lemma1}, given any tree $S_i$ with $i>0$,
the tree $S_{i-1}$ can be obtained in time $O(m(G))$.
Altogether, the stated running time follows,
which completes the proof.
\end{proof}
Theorem \ref{theorem1} now follows immediately
by combining Lemma \ref{lemma4} with Lemma \ref{lemma2}, 
choosing $\rho'$ equal to $2\rho$ for the latter.
Note that, since the tree $S$ produced by Lemma \ref{lemma4} 
intersects every bag of the tree decomposition, 
we have $d_G(u,V(S))\leq 2\rho$ for every vertex $u$ of $G$.

\end{document}